\documentclass{amsart}
\usepackage{amsmath}
\usepackage{amsthm}
\usepackage{amsfonts,mathrsfs}
\usepackage[flushleft]{threeparttable}
\usepackage{amssymb}
\usepackage{graphicx}
\usepackage{enumitem}
\usepackage{color}
\usepackage{verbatim}
\usepackage[foot]{amsaddr}

\theoremstyle{plain}
\newtheorem{theorem}{Theorem}[section]
\newtheorem{lemma}[theorem]{Lemma}

\newtheorem{corollary}[theorem]{Corollary}

\theoremstyle{definition}

\newtheoremstyle{TheoremNum}
	{\topsep}{\topsep}              
  {\itshape}                      
  {}                              
  {\bfseries}                     
  {.}                             
  { }                             
  {\thmname{#1}\thmnote{ \bfseries #3}}
\newtheorem{remark}{Remark}

\newcommand{\F}{\mathbb F}
\newcommand{\K}{\mathbb K}
\newcommand{\Z}{\mathbb Z}

\newcommand{\cN}{\mathcal N}

\newcommand{\cG}{\mathcal G}

\newcommand{\cH}{\mathcal H}
\newcommand{\cS}{\mathcal S}

\newcommand{\cT}{\mathcal T}

\newcommand{\bv}{\mathbf v}

\newcommand{\cC}{\mathcal C}
\newcommand{\U}{\mathcal U}
\newcommand{\rC}{\mathscr C}

\newcommand{\Aut}{\mathrm{Aut}}
\newcommand{\rank}{\mathrm{rk}}
\newcommand{\End}{\mathrm{End}}

\newcommand{\GL}{\mathrm{GL}}

\newcommand{\lp}[2]{\mathscr{L}_{(#1,#2)}}
\newcommand{\RN}[1]{%
  \textup{\uppercase\expandafter{\romannumeral#1}}%
}
\newcommand{\rn}[1]{%
  \textup{\lowercase\expandafter{\romannumeral#1}}%
}

 \def\zhou#1 {\fbox {\footnote {\ }}\ \footnotetext { From Yue: {\color{red}#1}}}

 \def\rocco#1 {\fbox {\footnote {\ }}\ \footnotetext { From Rocco: {\color{blue}#1}}}

\begin{document}
	\title[Nuclei and automorphism groups]{Nuclei and automorphism groups of generalized twisted Gabidulin codes}
	\author[R. Trombetti]{Rocco Trombetti\textsuperscript{\,1}}
	\address{\textsuperscript{1}Dipartimento di Mathematica e Applicazioni ``R. Caccioppoli", Universit\`{a} degli Studi di Napoli ``Federico \RN{2}", I-80126 Napoli, Italy}
	\email{rtrombet@unina.it}
	\author[Y. Zhou]{Yue Zhou\textsuperscript{\,2,$\dagger$}}
	\address{\textsuperscript{2}College of Liberal Arts and Sciences, National University of Defense Technology, 410073 Changsha, China}
	\address{\textsuperscript{$\dagger$}Corresponding Author}
	\email{yue.zhou.ovgu@gmail.com}
	\date{\today}
	
	\begin{abstract}
		Generalized twisted Gabidulin codes are one of the few known families of maximum rank metric codes over finite fields. As a subset of $m\times n$ matrices, when $m=n$, the automorphism group of any generalized twisted Gabidulin code has been completely determined by the authors in \cite{lunardon_generalized_2018}. In this paper, we consider the same problem for $m<n$. Under certain conditions on their parameters, we determine their middle nuclei and right nuclei, which are important invariants with respect to the equivalence for rank metric codes. Furthermore, we also use them to derive necessary conditions on the automorphisms of generalized twisted Gabidulin codes.
	\end{abstract}
	\keywords{Rank-metric code; maximum rank-distance code; automorphism; semifield; nucleus}
	\subjclass[2010]{51E20, 05B25, 51E22}
\maketitle

\section{Introduction}
Let $\K$ be a field. The set $\K^{m\times n}$ of all $m\times n$ matrices over $\K$ is a $\K$-vector space. The \emph{rank-metric distance} on $\K^{m\times n}$ is defined by 
\[d(A,B)=\mathrm{rk}(A-B) \,\, \text{for} \,\, A,B\in \K^{m\times n},\]
where $\rank(C)$ stands for the rank of $C$.

A subset $\cC\subseteq \K^{m\times n}$ is called a \emph{rank metric code}. The \emph{minimum distance} of $\cC$ is
\[d(\cC)=\min_{A,B\in \cC, A\neq B} \{d(A,B)\}.\]
When $\cC$ is a $\K$-linear subspace of $\K^{m\times n}$, we say that $\cC$ is a $\K$-linear code and its dimension $\dim_{\K}(\cC)$ is defined to be the dimension of $\cC$ as a subspace over $\K$.

Two rank metric codes $\cC_1$ and $\cC_2\subseteq \K^{m\times n}$ are \textit{equivalent} if there are $A\in \GL(m,\K)$, $B\in \GL(n,\K)$, $C\in \K^{m\times n}$ and $\gamma \in \Aut(\K)$ such that 
\begin{equation}\label{eq:equivalence_def}
	\cC_2=\{AX^{\gamma}B + C \mid X \in \cC_1\}.
\end{equation}  
In particular, when $\cC_1$ and $\cC_2$ are $\K$-linear, we can always let $C$ to be the zero matrix. All the equivalence mappings of a rank metric code $\cC$ form its \emph{automorphism group}, which is denoted by $\Aut(\cC)$.

There is another equivalence relation on rank metric codes called \emph{isometry} introduced in  \cite{de_la_cruz_algebraic_2016}. When $m\neq n$, the equivalence of rank metric codes is the same as the isometry. However, when $m=n$, we say that $\cC_1$ is \emph{isometric} to $\cC_2$ if $\cC_1$ is equivalent to $\cC_2$ or $\cC_2^{\top}$, where  $\cC^{\top}_2:=\{X^t \mid X \in \cC_2\},$ is the so called \emph{adjoint code} of $\cC_2$.

In this article we will be interested in the case $\K=\F_{q}$. Let $\cC$ be a rank metric code in $\F_q^{m\times n}$. When $d(\cC)=d$, it is well-known that
$$\#\cC\le q^{\max\{m,n\}(\min\{m,n\}-d+1)},$$
which is a $q$-analog of the Singleton bound for the rank metric distance; see \cite{delsarte_bilinear_1978}. When equality holds, we call $\cC$ a \emph{maximum rank distance} (MRD for short) code.

The main application of rank metric codes is in the construction of error correcting codes for {\it random network coding} \cite{koetter_coding_2008}. Nonetheless, there are several interesting structures in finite geometry, such as quasifields, splitting dimensional dual hyperovals and maximum scattered subspaces, which can be equivalently described as rank metric codes; see \cite{csajbok_maximum_2017,dempwolff_orthogonal_2015,johnson_handbook_2007,lunardon_kernels_2017,sheekey_rank-metric_arxiv,yoshiara_dimensional_2006}. For instance, the spreadset derived from a quasifield of order $q^n$ is an MRD code in $\F_q^{n\times n}$ and its minimum distance is $n$.

For MRD codes with minimum distance less than $\min\{m,n\}$, there are a few known constructions. The first and most famous family is due to Gabidulin \cite{gabidulin_MRD_1985} and Delsarte \cite{delsarte_bilinear_1978} who found it independently. This family is later generalized by Kshevetskiy and Gabidulin in \cite{kshevetskiy_new_2005}, and we often call them \emph{generalized Gabidulin codes}. 

Recent constructions of MRD codes can be found in \cite{cossidente_non-linear_2016,csajbok_newMRD_2017,csajbok_new_maximum_2018,csajbok_maximum_4_2018,durante_nonlinear_MRD_2016,neri_genericity_2018,otal_explicit_2016,otal_additive_2016,sheekey_new_2016,sheekey_new_arxiv,trombetti_new_toappear}. For instance in \cite{sheekey_new_2016}, the author exhibits two infinite families of linear MRD codes which are not equivalent to generalized Gabidulin codes. We call them {\it twisted Gabidulin codes} and {\it generalized twisted Gabidulin} codes. In \cite{lunardon_generalized_2018} it is shown that the latter family contains both generalized Gabidulin codes and twisted Gabidulin codes as proper subsets. Also in \cite{lunardon_generalized_2018}, when $m=n$, the automorphism groups and the equivalence issue for the generalized twisted Gabidulin codes have been completely solved.

In \cite{sheekey_new_2016}, generalized twisted Gabidulin codes are exhibited as a set of {\it linearized polynomials} over $\F_{q^n}$; i.e. as  a subset of the set of polynomials defined as follows: \[\lp{n}{q}[X]=\left\{\sum c_i X^{q^i}: c_i\in \F_{q^n} \right\}.\] 

More precisely, let  $n,k,s,h\in \Z^+$ with $k<n$ and $\gcd(n,s)=1$, and let $\eta$ be in $\F_{q^n}$ such that $N_{q^{sn}/q^s}(\eta)\neq (-1)^{nk}$. A generalized twisted Gabidulin code is the following set of linearized polynomials
\[\cH_{k,s}(\eta, h) = \{a_0 X + a_1 X^{q^s} + \dots +a_{k-1} X^{q^{s(k-1)}} + \eta a_0^{q^h} X^{q^{sk}}: a_0,a_1,\dots, a_{k-1}\in \F_{q^n} \}.\] Indeed each polynomial in this set has at most $q^{k-1}$ roots in $\F_{q^n}$; see \cite{lunardon_generalized_2018,sheekey_new_2016}.

Any polynomial $f$ in $\lp{n}{q}[X]$ gives rise to an $\F_q$-linear map $x\in \F_{q^n} \mapsto f(x) \in \F_{q^n},$ and it is well known that $(\lp{n}{q}[X]/(X^{q^n}-X),+,\circ,\cdot),$ where $+$ is the addition of maps, $\circ$ is the composition of maps and $\cdot$ is the scalar multiplication by elements of $\F_q$, is isomorphic to the algebra of $n\times n$ matrices over $\F_q$ and to $\End_{\F_q}(\F_{q^n})$ which denotes the set of $\F_q$-endomorphisms of $\F_{q^n}$.

Now, let $m\le n$ and let $S=\{\alpha_1,\dots,\alpha_m\}$ be a set made up of $m$ linearly independent elements of $\F_{q^n}$ over $\F_q$. Under a given basis of $\F_{q^n}$ over $\F_q$, each element $a$ of $\F_{q^n}$ can be written as a (column) vector $\bv(a)$ in $\F_{q}^n$. Most of MRD codes with $1<k<n-1$  appeared in the literature so far, are in the following form:
\begin{equation}\label{eq:mn_MRD_cH}
\left\{ \left(\bv(f(\alpha_1)), \dots, \bv(f(\alpha_m))\right)^T: f\in \cH_{k,s}(\eta, h)
  \right\},
\end{equation}
where $(\cdot )^T$ denotes the transpose of a matrix. In this regard, we point out that several new constructions of MRD codes which are not in this form are presented recently in \cite{horlemann-trautmann_new-criteria_2017} and are proved to be inequivalent to any Gabidulin code. However, we do not know whether they are equivalent or not to a generalized twisted Gabidulin code \eqref{eq:mn_MRD_cH}.

When $m<n$, MRD codes defined in \eqref{eq:mn_MRD_cH} can be seen as the image of $\cH_{k,s}(\eta, h)$ under a projection from $\F_q^{n\times n}$ to $\F_q^{m\times n}$. Indeed, let $\xi$ be a primitive element of $\F_{q^n}$. It is clear that 
\[\cH_{k,s}(\eta, h) \cong \left\{ \left(\bv(f(1)),\bv(f(\xi)), \dots, \bv(f(\xi^{n-1}))\right)^T: f\in \cH_{k,s}(\eta, h)
  \right\}.\]
As $\{1, \xi, \cdots, \xi^{n-1} \}$ is a basis of $\F_{q^n}$ over $\F_q$, there exist $a_{ij}\in \F_q$ such that
$\alpha_i = \sum_{j=1}^na_{ij}\xi^{j-1}$ for $i=1,2,\cdots, m$. Let $L=(a_{ij})$. By multiplying $L$ on the left side of elements in $\cH_{k,s}(\eta, h) $, we get \eqref{eq:mn_MRD_cH}. 

In general, it is difficult to tell whether two rank metric codes with the same parameters are equivalent or not. For quasifields, in particular for semifields, there are some classical invariants called {\it kernel}, {\it left}, {\it right} and {\it middle nuclei}. Originally, these are defined as algebraic substructures of quasifields or semifields. However they can also be translated into the language of matrices. For more information on the nuclei of finite semifields, we refer to \cite{marino_nuclei_2012}. These invariants are quite useful in telling the equivalence between two semifields, and many classification results on semifields are also based on certain assumptions on the sizes of their nuclei; see \cite{blokhuis_classification_2003,marino_nuclei_2012,marino_towards_2011,menichetti_kaplansky_1977,menichetti_n-dimensional_1996} for instance. 

It is natural to ask whether some of these substructures can be defined also for other rank metric codes. This is addressed in \cite{lunardon_kernels_2017}, where the kernel, the middle nucleus and the right nucleus are defined for an arbitrary rank metric code. It can proved that the order of all such structures is an invariant with respect to the equivalence relation for $\K$-linear rank metric codes in $\K^{m\times n}$.

In \cite{liebhold_automorphism_2016}, the middle nucleus and the right nucleus of an MRD code defined as in \eqref{eq:mn_MRD_cH} with $\eta=0$, i.e.\ when $\cH_{k,s}(\eta, h)= \cG_{k,s}$, are determined; see \cite{morrison_equivalence_2014} for the calculation of the middle nucleus too. This is a crucial step towards the determination of the automorphism group of these codes in  \cite{liebhold_automorphism_2016}. We point out that in \cite{liebhold_automorphism_2016} the middle nucleus (resp.\ right nucleus) is called the \emph{left idealiser}  (resp.\ \emph{right idealiser}) of the code.

For nonzero $\eta$, if $m=n$, the automorphism group of each generalized twisted Gabidulin code is determined in \cite{lunardon_generalized_2018}. When $m<n$, the determination of the automorphism group appears to be a more complicated task.  In a very recent paper \cite{csajbok_puncturing_2018}, the special case with $m\mid n$ was investigated by using cyclic models for bilinear forms and the automorphism groups are determined.

In this article we investigate the middle nucleus as well as the right nucleus of a generalized twisted Gabidulin code defined as in \eqref{eq:mn_MRD_cH} with $\eta \neq 0$, for $m < n$. Under certain assumptions on the involved parameters $k$, $s$, $m$, $\eta$ and $n$, we can determine these nuclei. Finally, by exploiting these results we can derive necessary conditions on the automorphisms of generalized twisted Gabidulin codes under certain restrictions for the involved parameters.

The rest of this paper is organized as follows: In Section \ref{sec:pre}, we introduce several notation used throughout this paper and translate an MRD code as a set of matrices over $\F_q$ into a set of linearized polynomials in $\F_{q^n}[X]$. In Section \ref{se:nuclei}, we calculate the middle nucleus as well as the right nucleus of a generalized twisted Gabidulin code under certain assumptions. In the end, we use these results to obtain some necessary conditions on the automorphisms of generalized twisted Gabidulin codes.

\section{Preliminaries}\label{sec:pre}
Throughout this paper, we always use $\cS$ to denote a set of $m$ $\F_q$-linearly independent elements in $\F_{q^n}$. Of course $m\leqslant n$. We often use $U_\cS$ to denote the $\F_q$-subspace generated by the elements of $\cS$ in $\F_{q^n}$. 


To investigate rank metric codes in $\F_{q}^{m\times n}$ where $m\leqslant n$, we need to prove several results.
\begin{lemma}\label{le:polynomials_matrices}
	Let $m$, $n$ be in $\Z^+$ satisfying $m\leqslant n$, and let $q$ be a prime power. Let $\cS$ be a subset consisting of $m$ arbitrary $\F_q$-linearly independent elements $\alpha_1,\dots,\alpha_m\in \F_{q^n}$. Define 
	$\theta_\cS:= \prod_{u\in U_\cS} (X-u)$.
	Then we have 
	\[\lp{n}{q}[X] / (\theta_\cS) \cong \left\{\left(\bv(f(\alpha_1)), \dots, \bv(f(\alpha_m))\right)^T: f\in  \lp{n}{q}[X] \right\} .\]
\end{lemma}
\begin{proof}
	The map given by
	\[\varphi : f \mapsto  \left(\bv(f(\alpha_1)), \dots, \bv(f(\alpha_m))\right)^T\]
	is clearly surjective and $\F_q$-linear. By noting that $\varphi(f)$ is the zero matrix if and only if $f(x)=0$ for every $x\in U_\cS$, we see that $\ker(\varphi) = \{f\, \in \lp{n}{q}[X]: f \equiv 0\mod{\theta_\cS}\}$. This concludes the proof.
\end{proof}
For the subset $\cS$ made up of $m$ arbitrary $\F_q$-linearly independent elements in $\F_{q^n}$, we define  
\begin{center}
\begin{tabular}{cccc}
	$\pi_\cS$ :& $\lp{n}{q}[X]$ &$\rightarrow$ &$\lp{n}{q}[X]/(\theta_\cS)$,\vspace{0.2cm}\\
		 & $f$            &$\mapsto$     & $f \mod{\theta_\cS}$.
\end{tabular}
\end{center}
In particular, for $U_\cS=\F_{q^n}$, by Lemma \ref{le:polynomials_matrices} we have
\[ \End_{\F_q}(\F_{q^n}) \cong  \lp{n}{q}[X]/(X^{q^n}-X),\]
where  $\End_{\F_q}(\F_{q^n})$ is the set of $\F_q$-endomorphisms of $\F_{q^n}$.

\begin{lemma}\label{le:f=0}
	Let $\cS$ be an $m$-subset $\cS$ of $\F_q$-linearly independent elements in $\F_{q^n}$.	Let $\rC$ be a subset of $\lp{n}{q}[X]$. 
	Assume that for any distinct $f$ and $g\in \rC$, the number of solutions of $f=g$ in $U_\cS$ is strictly smaller than $q^m$. Then $\pi_\cS$ is injective on $\rC$.
\end{lemma}
\begin{proof}
	It follows directly from the assumption $\#\{x\in \F_{q^n}: (f-g)(x)=0 \}<q^m=\#U_\cS$ and the fact that $h\equiv 0 \mod{\theta_\cS}$ if and only if $h(u)=0$ for every $u\in U_\cS$.
\end{proof}

By Lemma \ref{le:f=0} and the fact that  $\cG_{m,s}$ is an MRD code, the following result can readily be verified.
\begin{corollary}\label{coro:representation_all}
	Let $\cS$ be an $m$-subset $\cS$ of $\F_q$-linearly independent elements in $\F_{q^n}$. Let $s$ be an integer such that $\gcd(n,s)=1$. Then the set 
	\[\{a_0 X + a_1 X^{q^s} + \dots +a_{m-1} X^{q^{s(m-1)}}: a_i\in \F_{q^n} \}\]
	is a transversal, namely a system of distinct representatives, for the ideal $(\theta_\cS)$ in  $\lp{n}{q}[X]$.
\end{corollary}

Clearly, if an $\F_q$-linear subset $\rC\subseteq \lp{n}{q}[X]$ is  of size $q^{nk}$ and each nonzero polynomial in $\rC$ has at most $q^{k-1}$ roots over $U_\cS$, for instance $\rC= \cH_{k,s}(\eta, h)$, then the assumption on $\rC$ in Lemma \ref{le:f=0} is satisfied.

By Lemmas \ref{le:polynomials_matrices} and \ref{le:f=0}, codes described in \eqref{eq:mn_MRD_cH} can be equivalently written as 
\begin{equation}\label{eq:polynomials_GTGC}
		\pi_\cS(\cH_{k,s}(\eta, h)) = \{f \mod {\theta_\cS}: f\in \cH_{k,s}(\eta, h)\}\subseteq \lp{n}{q}[X]/(\theta_\cS).
\end{equation}
In particular, when $m=n$, it becomes 
$$\{f \mod (X^{q^n}-X): f\in \cH_{k,s}(\eta, h)\}\subseteq \lp{n}{q}[X]/(X^{q^n}-X).$$

In fact, for any subset $\cS$ of $\F_q$-linearly independent elements $\alpha_1, \cdots, \alpha_m$ and any subset $\rC$ of $\lp{n}{q}[X]$, we can always get a rank metric code
\begin{equation}\label{eq:general_code_from_poly}
	\cC=\left\{\left(\bv(f(\alpha_1)), \dots, \bv(f(\alpha_m))\right)^T: f\in  \rC \right\},
\end{equation}
which is an MRD code if and only if $\#\rC = q^{nk}$ and for any distinct $f,g\in  \rC$, $f-g$ has at most $q^{k-1}$ roots in $U_\cS$.

Let $\cC_1$ and $\cC_2  \in \F_q^{m \times n}$ be two rank metric codes. As explained in the previous section if $m\neq n$, isometry and equivalence between  $\cC_1$ and  $\cC_2$  are the same; otherwise $m=n$ and $\cC_1$ is isometric to $\cC_2$ if $\cC_1$ is equivalent to $\cC_2$ or $\cC_2^{\top}$. All the equivalence maps of a rank metric code $\cC$ form its automorphism group, which is denoted by $\Aut(\cC)$.

It is straightforward to verify the following lemma.
\begin{lemma}\label{le:automophism_group}
	Let $\rC$ be a subset of $\lp{n}{q}[X]$, and let $\cS$ be an $m$-subset $\cS$ of $\F_q$-linearly independent elements in $\F_{q^n}$. The automorphism group $\Aut(\cC)$ of $\cC$ defined by \eqref{eq:general_code_from_poly} is isomorphic to the group $\Aut(\pi_\cS(\rC))$ consisting of $(\varphi_1,\varphi_2|_{U_\cS},\rho)$, where  $\rho\in \Aut(\F_{q})$ and $\varphi_1,\varphi_2\in \End_{\F_q}(\F_{q^n})$ satisfy 
	\begin{enumerate}[label=(\alph*)]
		\item 	$\varphi_1(\F_{q^n})=\F_{q^n}$,
		\item	$\varphi_2(U_\cS) = U_\cS$,
		\item	$\pi_\cS(\varphi_1\circ f^\rho\circ \varphi_2)\in \pi_\cS(\rC)  $, for all $f \in \rC$. 
	\end{enumerate}
	Here $\varphi_2|_{U_\cS} = \varphi'_2|_{U_\cS}$ if and only if $\varphi_2(x)=\varphi'_2(x)$ for every $x\in U_\cS$.
\end{lemma}

In the remaining part of this article, we will frequently use the same notation, for example $\varphi$, to denote an element in $\End_{\F_q}(\F_{q^n})$ as well as its image in $\lp{n}{q}[X]/(X^{q^n}-X)$ under a certain isomorphism. Just as in Lemma \ref{le:automophism_group}, $\varphi_1\circ f^\rho\circ \varphi_2$ is a composition of endomorphisms in $\End_{\F_q}(\F_{q^n})$, and it is also a linearized polynomial which we consider as its image under $\pi_\cS$.

\section{Nuclei of generalized twisted Gabidulin codes}\label{se:nuclei}

Let $\K$ be an arbitrary field. Let $\cC \subseteq \K^{m\times n}$ be a $\K$-linear rank metric code. We define the {\it middle nucleus} of $\cC$ as the following set of matrices of order $m$: 
$$N_m(\cC) = \{Z \in \K^{m\times m} \, \colon \, ZC \in \cC \text{ for all } C \in \cC \}.$$
In the same way we say that the {\it right nucleus} of $\cC$ is the following set: $$N_r(\cC)= \{Y \in \K^{n\times n} \, : \, CY \in \cC  \text{ for all }C \in \cC \}.$$

In particular, when $\cC$ defines a finite semifield $\mathbb{S}$, $N_m(\cC)$ (resp.\ $N_r(\cC)$) is exactly the middle (resp.\ right) nucleus of $\mathbb{S}$. 
In \cite{liebhold_automorphism_2016}, they are called the \emph{left idealiser} and the \emph{right idealiser} of $\cC$, respectively.

In \cite{liebhold_automorphism_2016}, Liebhold and Nebe computed the middle nucleus and the right nucleus of 
\[	\left\{ (\bv(f(\alpha_1), \dots, \bv(f(\alpha_m))^T: f\in \cG_{k,s}  \right\},\]
for any $m\leqslant n$ and any $\F_q$-linearly independent set $\{\alpha_1,\cdots,\alpha_m \}$. In \cite{lunardon_kernels_2017}, the middle nucleus and the right nucleus of \eqref{eq:polynomials_GTGC} were determined in the special case $m=n$.

In this section, we proceed to investigate the middle nucleus and the right nucleus of \eqref{eq:mn_MRD_cH} for $m<n$ and any $\F_q$-linearly independent set $\{\alpha_1,\cdots,\alpha_m \}$.

We need several lemmas for our final results.
\begin{lemma}\cite[Theorem 5.4]{lunardon_kernels_2017}\label{le:nuclei_of_MRD}
	Let $\cC$ be a linear MRD code in $\F_q^{m\times n}$ with $m\leqslant n$ and $d>1$. Then the following statements hold:
	\begin{enumerate}[label=(\alph*)]
	\item Its middle nucleus $N_m(\cC)$ is a finite field.
	\item When $\max\{d,m-d+2\} \geqslant \left\lfloor \frac{n}{2} \right\rfloor +1$, its right nucleus $N_r(\cC)$ is a finite field.
	\end{enumerate}
\end{lemma}

\begin{lemma}\cite[Lemma 3.2]{liebhold_automorphism_2016}\label{le:liebhold}
	Let $B$ be an $\F_q$-basis of $\F_{q^n}$. Let $\Delta_B(\F_{q^n})$ denote the regular representation of $\F_{q^n}$ in $\F_q^{n\times n}$ with respect to $B$, i.e.\ the set of matrices associated to the $\F_q$-linear maps $x\mapsto ax$ for $a,x\in \F_{q^n}$. Let $A$ be an $\F_q$-algebra satisfying
	\[\Delta_B(\F_{q^n}) \subseteq  A \subseteq \F_q^{n\times n}. \]
	Then there is a subfield $\F_q\subseteq \F_{q^\ell}\subseteq \F_{q^n}$ such that
	\[A= C_{\F_q^{n\times n}} (\Delta_B(\F_{q^n})) \cong \F_{q^\ell}^{r\times r},\]
	where $r=n/\ell$ and $C_{\F_q^{n\times n}} (M)$ stands for the centralizer of $M$ in $\F_q^{n\times n}$.
\end{lemma}

By \eqref{eq:polynomials_GTGC} and the definition of right nucleus, it is routine to verify the following statement.
\begin{lemma}\label{le:nucleus_polynomial}
	Let $\cS $ be an $m$-subset of $\F_q$-linearly independent elements in $\F_{q^n}$. For any $\rC\subseteq\lp{n}{q}[X]$, the right nucleus of the corresponding rank metric code defined by \eqref{eq:general_code_from_poly} is isomorphic to 
	\[\cN_r(\pi_\cS(\rC)) = \left\{\varphi\in \End_{\F_q}(\F_{q^n}) : \pi_\cS(\varphi\circ f)\in  \pi_\cS(\rC)\text{ for all }f\in \rC\right\}. \] 
	Its middle nucleus is  isomorphic to 
	\[\cN_m(\pi_\cS(\rC)) = \left\{\psi|_{U_\cS}: \psi\in \End_{\F_q}(\F_{q^n}) , \psi(U_\cS)\subseteq U_\cS, \pi_\cS(f\circ\psi)\in  \pi_\cS(\rC)\text{ for all }f\in \rC\right\}.\]
\end{lemma}

When $\pi_\cS(\rC)$ defines an MRD code $\cC$ by \eqref{eq:general_code_from_poly}, by Lemma \ref{le:nuclei_of_MRD}, all nonzero elements in $\cN_m(\pi_\cS(\rC))$ are invertible, which means that they form a subgroup of the automorphism group $\Aut(\pi_\cS(\rC))$; however, as the nonzero elements in $\cN_r(\pi_\cS(\rC))$ are not always invertible in general, it is not necessary that they form a subgroup in $\Aut(\pi_\cS(\rC))$.

\begin{lemma}\label{le:extra_important}
	Let $\cS$ be an $m$-subset $\cS$ of $\F_q$-linearly independent elements in $\F_{q^n}$. Let $s$ be an integer such that $\gcd(n,s)=1$ and $\varphi\in\lp{n}{q}[X]$. For any $a\in \F_{q^n}$, assume that 
	$$\varphi(aX)\equiv\sum_{j=0}^{m-1} c_j X^{q^{js}} \mod \theta_\cS,$$
	where $c_j$ for $j=0,\cdots, m-1$ depend on $a$. Let $j_0:=\max \{j : c_j\neq0\}$. 
	For any non-negative integer $t$ satisfying $t+j_0\leq m-1$ and any $j$, $c_j\neq 0$ if and only if there exists $\bar{a}\in \F_{q^n}$ such that  
	$$\varphi(\bar{a}X^{q^{ts}})\equiv\sum_{j=0}^{m-1} \bar{c}_j X^{q^{js}} \mod \theta_\cS$$
	and $\bar{c}_{j+t}\neq 0$.
\end{lemma}
\begin{proof}
	By Corollary \ref{coro:representation_all}, we know that for each $i\in\{0,\cdots, n-1 \}$, there exist $e_i^{(j)}\in \F_{q^n}$ for $j=0,1,\cdots, m-1$ such that
	\begin{equation}\label{eq:e_ij_expand}
		X^{q^i} \equiv \sum_{j=0}^{m-1} e_i^{(j)} X^{q^{js}} \mod \theta_\cS.
	\end{equation}
	Without loss of generality, we assume that $\deg(\varphi)\leq q^{n-1}$ and $\varphi=\sum_{i=0}^{n-1}d_i X^{q^i}$. By \eqref{eq:e_ij_expand}, 
	\begin{align*}
		\varphi(aX) &\equiv \sum_{i=0}^{n-1} d_i a^{q^i} \sum_{j=0}^{m-1} e_i^{(j)} X^{q^{js}} \mod \theta_\cS\\
					&\equiv \sum_{j=0}^{m-1}\left(\sum_{i=0}^{n-1} d_i a^{q^i}  e_i^{(j)}\right) X^{q^{js}} \mod \theta_\cS,
	\end{align*}
	for any $a\in \F_{q^n}$. For any given $j$, there exists $a\in \F_{q^n}^*$ such that $c_j:=\sum_{i=0}^{n-1} d_i a^{q^i}  e_i^{(j)}\neq 0$ if and only if there is $i_0\in \{0,\cdots, n-1\}$ such that $d_{i_0}e_{i_0}^{(j)}\neq 0$.
	On the other hand,
	\begin{align*}
		\varphi(\bar{a}X^{q^{ts}}) &\equiv \sum_{i=0}^{n-1} d_i \bar{a}^{q^i} \left(\sum_{j=0}^{m-1} e_i^{(j)} X^{q^{js}}\right)^{q^{ts}} \mod \theta_\cS\\
					&\equiv \sum_{j=0}^{m-1}\left(\sum_{i=0}^{n-1} d_i \bar{a}^{q^i}  e_i^{(j)q^{ts}}\right) X^{q^{(j+t)s}} \mod \theta_\cS.
	\end{align*}
	Thus, for any $t$ satisfying $t+j_0\leq m-1$, $\bar{c}_{j+s}:=\sum_{i=0}^{n-1} d_i \bar{a}^{q^i}  e_i^{(j)q^{ts}}\neq 0$ for some $\bar{a}$ if and only if there is at least one term $d_i  e_i^{(j)}\neq 0$, which is equivalent to $c_j\neq 0$ for some $a\in \F_{q^n}^*$.  This finishes the proof.
\end{proof}
For any $\varphi \in \lp{n}{q}[X]$, $\theta_\cS$ defined as in Lemma \ref{le:polynomials_matrices} and $c_j$ and $\bar{c}_j$ defined as in Lemma \ref{le:extra_important}, if we define
\[\mathscr{A}_{\varphi,0}:=\{ j: c_j\neq 0 \text{ for some }a\in \F_{q^n}\}\]
and
\[\mathscr{A}_{\varphi,t}:=\{ j: \bar{c}_{j}\neq 0 \text{ for some }\bar{a}\in \F_{q^n}\},\]
then Lemma \ref{le:extra_important} shows us that $\mathscr{A}_{\varphi,t}=\{i+t: i\in \mathscr{A}_{\varphi,0}\}$ for $t\leq m-1-j_0$ where $j_0=\max\mathscr{A}_{\varphi,0}$.

Next, we proceed to show that for most cases, monomials in $\pi_\cS(\cH_{k,s}(\eta, h))$ are mapped to monomials by the elements in its right or middle nucleus.
\begin{lemma}\label{le:mono_to_mono_right}
	Let $k$, $m$ and $n$ be positive integers satisfying $k<m\leqslant n$. Let $\cS $ be an $m$-subset of $\F_q$-linearly independent elements in $\F_{q^n}$.	For each element $\varphi\in\cN_r( \pi_\cS(\cH_{k,s}(\eta, h)))$ and any $a\in \F_{q^n}$, there exists an element $b\in \F_{q^n}$ such that
	$$\varphi(aX) \equiv bX \mod \theta_\cS,$$
	if one of the following collections of conditions are satisfied.
	\begin{enumerate}[label=(\alph*)]
	\item $\eta=0$, i.e., $\cH_{k,s}(\eta, h)=\cG_{k,s}$.
	\item $\eta\neq 0$, $m>k+1$ and $(m,k)\neq(4,2)$. 
	\end{enumerate}
\end{lemma}
\begin{proof}
	By Lemma \ref{le:nucleus_polynomial}, for each $\varphi \in  \cN_r(\pi_\cS(\cH_{k,s}(\eta, h)))$,
	\begin{equation}\label{eq:varphi(ax)}
		\varphi(aX) = \sum_{i=0}^{n-1} d_i (aX)^{q^i}, 
	\end{equation}
	for certain $d_i\in \F_{q^n}$.
	
	\vspace*{2mm}
	\noindent(a) Recall that when $\eta=0$,
	\[\cG_{k,s}=\cH_{k,s}(0, h) = \{a_0 X + a_1 X^{q^s} + \dots +a_{k-1} X^{q^{s(k-1)}} : a_0,a_1,\dots, a_{k-1}\in \F_{q^n} \}.\]
	As $\pi_\cS(\varphi(aX))\in \pi_\cS(\cG_{k,s})$, we assume that 
	$$\varphi(aX)\equiv\sum_{i=0}^{k-1} c_i X^{q^{is}} \mod \theta_\cS,$$ 
	where $c_i$ depend on the value of $a$; see \eqref{eq:varphi(ax)}. Let $i_0:= \max\{i:c_i\neq 0\}$. We are going to prove that $i_0 = 0$.  
	
	For any integer $j\geq 0$ satisfying $j+i_0\leq m-1$, by Lemma \ref{le:extra_important}, there exists $\bar{a}\in \F_{q^n}$ such that
	\begin{equation}\label{eq:varphi(ax)_1}
		\varphi(\bar{a}X^{q^{js}})\equiv \sum_{i=0}^{m-1} \bar{c}_i X^{q^{(i+j)s}} \mod \theta_\cS,
	\end{equation}
	where $\bar{c}_{i_0}\neq 0$.
	
	By way of contradiction, we assume that $i_0>0$. It means $0\leq k-i_0\leq m-1 $ whence $\bar{a}X^{q^{(k-i_0)s}}\in \cG_{k,s}$. Plugging $j=k-i_0$ into \eqref{eq:varphi(ax)_1}, we have
	\begin{equation}\label{eq:re_add_1}
		\varphi(\bar{a}X^{q^{(k-i_0)s}})\equiv \sum_{i=0}^{m-1} \bar{c}_i X^{q^{(i+k-i_0)s}} \mod \theta_\cS.
	\end{equation}
	From $\bar{a}X^{q^{(k-i_0)s}}\in \cG_{k,s}$ and  $\bar{c}_i X^{q^{(i+k-i_0)s}}\in \cG_{k,s}$ for all $0\leq i<i_0$, we can derive
	$\pi_\cS(c_{i_0} X^{q^{ks}})\in\pi_\cS(\cG_{k,s})$. On the other hand, by Corollary \ref{coro:representation_all} any $g\in \cG_{k,s}$ and $c_{i_0} X^{q^{ks}}$ belong to distinct residue classes in $ \lp{n}{q}[X] / (\theta_\cS)$, from which it follows that $\pi_\cS(c_{i_0} X^{q^{ks}})\notin\pi_\cS(\cG_{k,s})$. It is a contradiction.
	
%
	Hence, $i_0=0$ which concludes the proof.
	
	\vspace*{2mm}
	\noindent (b) Depending on the value of $k$, we divide our proof into three cases.
	\begin{enumerate}[label=(\roman*)]
		\item When $k=1$ and $m>k+1=2$, it is not difficult to show that for every $a\in \F_{q^n}$, $\varphi(aX) \equiv bX \mod \theta_\cS$ for some $b\in \F_{q^n}$. Precisely since $\pi_\cS(\varphi(aX))\in \pi_\cS(\cH_{1,s}(\eta, h))$, we may assume, by way of contradiction, that there exist $b_0$, $b_1\in \F_{q^n}$ with $b_1\neq 0$ such that 
		\[\varphi(aX) \equiv b_0X + b_1 X^{q^s} \mod \theta_\cS,\]
		because of $\pi_\cS(\varphi(aX))\in \pi_\cS(\cH_{1,s}(\eta, h))$. Furthermore, by Lemma \ref{le:extra_important}, we can choose $c$ such that
		\[\varphi(cX + \eta c^{q^h}X^{q^s}) \equiv b'_0X + b'_1 X^{q^s} + b'_2 X^{q^{2s}}\mod \theta_\cS,\]
		where $b'_2\neq 0$. It follows that $\pi_\cS( b'_0X + b'_1 X^{q^s} + b'_2X^{q^{2s}})\in \pi_\cS(\cH_{1,s}(\eta, h))$. However, by $m>2$ and Corollary \ref{coro:representation_all}, there cannot exist any $a_0\in \F_{q^n}$ such that
		\[a_0X +  \eta a_0^{q^h}X^{q^{s}} \equiv  b'_0X + b'_1 X^{q^s} + b'_2X^{q^{2s}}\mod \theta_\cS,\]
		which is a contradiction.
		\item 	When $k>2$ and $m>k+1$, by an analogous argument as in the proof of (a), we can show that 
		\begin{equation}\label{eq:varphi(ax)_2}
			\varphi(aX^{q^s})\equiv c_{0} X + c_1 X^{q^{s}} + c_2 X^{q^{2s}}  \mod \theta_\cS,
		\end{equation}
		for certain $c_i\in \F_{q^n}$. Otherwise, suppose that $\varphi(aX^{q^s})$ has more nonzero coefficients which means $i_0:= \max\{i :c_i\neq 0\} > 2$. As $0 < k+2-i_0 < k$, we have that $\bar{a}X^{q^{s(k+2-i_0)}}\in \cH_{k,s}(\eta, h)$ for all $\bar{a}\in\F_{q^n}$. As in \eqref{eq:re_add_1}, by looking at $\varphi(\bar{a}X^{q^{s(k+2-i_0)}})$, we derive  $\pi_{\cS}(\bar{a}X^{q^{s(k+1)}}) \in \pi_{\cS}(\cH_{k,s}(\eta, h))$ for all $\bar{a}\in\F_{q^n}$, which, since $m>k+1$, leads to a contradiction.
		
		We proceed to show that $c_0$ must be $0$ by way of contradiction. It can be similarly shown that $c_2=0$, and we omit its proof.
		
		As $k>2$, we have $c_2X^{q^{2s}}\in  \cH_{k,s}(\eta, h)$. Noting that $c_1X^{q^{s}},\varphi(aX^{q^s})\in  \cH_{k,s}(\eta, h)$, together with \eqref{eq:varphi(ax)_2}, i.e.,
		\[c_{0} X\equiv \varphi(aX^{q^s})-c_1 X^{q^{s}} - c_2 X^{q^{2s}}  \mod \theta_\cS ,\]
		we have $\pi_\cS(c_{0} X) \in \pi_\cS(\cH_{k,s}(\eta, h))$. However, by Corollary \ref{coro:representation_all}, for arbitrary $g\in \cH_{k,s}(\eta, h)$, $c_{0} X$ and $g$ belong to distinct residue classes in $ \lp{n}{q}[X] / (\theta_\cS)$, which is a contradiction.  
		
		Therefore $\varphi(aX^{q^s})\equiv c_1 X^{q^{s}} \mod \theta_\cS$, which implies that $\varphi(\bar{a}X) \equiv bX \mod \theta_\cS$ for certain $b\in \F_{q^n}$ by Lemma \ref{le:extra_important}.
		
		\item 	When $k=2$ and $m>k+2=4$, again we can derive  \eqref{eq:varphi(ax)_2} for certain $c_i\in \F_{q^n}$. Next we prove that $c_0$ must be $0$ by way of contradiction and it can be similarly shown that $c_2=0$. As $\varphi(aX^{q^s})\equiv c_1 X^{q^{s}} \mod \theta_\cS$ implies that $\varphi(aX) \equiv bX \mod \theta_\cS$ for certain $b\in \F_{q^n}$ by Lemma \ref{le:extra_important}, we complete the proof.
		
		Now assume that $c_0\neq 0$. By composing $X\mapsto X^{q^{n-s}}$ with each element of the transversal in Corollary \ref{coro:representation_all}, we obtain another transversal
		\[\{a_{-1} X^{q^{n-s}} + a_0 X + a_1 X^{q^s} + \dots +a_{m-2} X^{q^{s(m-2)}}: a_i\in \F_{q^n} \}\]
		for the ideal $(\theta_\cS)$ in  $\lp{n}{q}[X]$. By an analogous argument as in the proof of Lemma \ref{le:extra_important}, from \eqref{eq:varphi(ax)_2} we can show that there always exists $\bar{a}\in \F_{q^n}$ such that
		\[\varphi(\bar{a} X)  \equiv \bar{c}_0 X^{q^{n-s}} + \bar{c}_1 X + \bar{c}_2X^{q^{s}} \mod \theta_\cS,\]
		where $\bar{c}_0\neq 0$. Similarly we can also expand $\varphi(\eta \bar{a}^{q^{h}} X^{q^{2s}})$, sum it with $\varphi(\bar{a} X)$ and we get
		\[\varphi(\bar{a} X + \eta \bar{a}^{q^{h}} X^{q^{2s}})\equiv \bar{c}_0 X^{q^{n-s}} + d_0 X + d_1X^{q^{s}} +d_2 X^{q^{2s}}+d_3 X^{q^{3s}}\mod \theta_\cS,  \] 
		where $\bar{c}_0\neq 0$ and $d_0, d_1$, $d_2$ and $d_3\in \F_{q^n}$. Let $f$ denote $\varphi(\bar{a} X + \eta \bar{a}^{q^{h}} X^{q^{2s}})$. As $\pi_\cS(f)\in \pi_\cS(\cH_{2,s}(\eta, h))$, there exists a polynomial $g\in \cH_{2,s}(\eta, h)$ such that
		\[f - g \equiv 0 \mod{\theta_\cS}.\]
		However, as $(f - g)^{q^s}$ modulo $\theta_\cS$ is congruent to a $q^s$-polynomial of degree smaller than or equal to $q^{4s}$,  $f-g$ has at most $q^{4}$ roots in $\U_\cS$ if $f-g\neq 0$ (recall that each polynomial $\cH_{k,s}(\eta, h)$ has at most $q^{k-1}$ roots), which actually holds because of $\bar{c}_0\neq 0$. It contradicts the fact that $\theta_\cS$ has $q^m$ roots and $\theta_\cS \mid (f-g)$.\qedhere
	\end{enumerate}
\end{proof}

For the middle nucleus of $\pi_\cS(\cH_{k,s}(\eta, h))$, we can prove a result analogous to Lemma \ref{le:mono_to_mono_right}.

\begin{lemma}\label{le:mono_to_mono_middle}
	Let $k$, $m$ and $n$ be positive integers satisfying $k<m\leqslant n$, and let $\eta, \tilde{\eta}\in \F_{q^n}$ and $h,\tilde{h}\in \{0,\cdots n-1\}$ be such that $\cH_{k,s}(\eta, h)$ and $\cH_{k,s}(\tilde{\eta}, \tilde{h})$ are both generalized twisted Gabidulin codes.  Let $\cS $ be an $m$-subset of $\F_q$-linearly independent elements in $\F_{q^n}$. Assume that $\psi\in \End_{\F_{q}}(\F_{q^n})$ satisfies  $\pi_\cS(f\circ \psi )\in \pi_\cS(\cH_{k,s}(\tilde{\eta}, \tilde{h}))$ for all $f\in \cH_{k,s}(\eta, h)$. Then there exists an element $b\in \F_{q^n}$ such that
	$$\psi(X) \equiv bX \mod \theta_\cS,$$
	if one of the following collections of conditions are satisfied.
	\begin{enumerate}[label=(\alph*)]
	\item $\eta=\tilde{\eta}=0$, i.e., $\cH_{k,s}(\eta, h)=\cH_{k,s}(\tilde{\eta}, \tilde{h})=\cG_{k,s}$.
	\item $\eta,\tilde{\eta}\neq 0$, $m>k+1$ and $(m,k)\neq(4,2)$. 
	\end{enumerate}
\end{lemma}
\begin{proof}
	Assume that
	\begin{equation}	\label{eq:avarphi(x)}
		\psi(X) = \sum_{i=0}^{n-1} e_i X^{q^i},
	\end{equation}
	for certain $e_i\in \F_{q^n}$.
	
        \vspace*{2mm}
	\noindent(a) 
	As $\pi_\cS(\psi(X))\in \pi_\cS(\cG_{k,s})$, we may assume that $$\psi(X)\equiv\sum_{i=0}^{k-1} c_i X^{q^{is}} \mod \theta_\cS.$$ 
 Then for any integer $j$, we have 
	\begin{equation}\label{eq:psi(ax)_1}
		\psi(X)^{q^{js}}\equiv \sum_{i=0}^{k-1} c_i^{q^{js}} X^{q^{(i+j)s}} \mod \theta_\cS.
	\end{equation}

	Assume that there is at least one $c_i\neq 0$ and let $i_0:= \max\{i:c_i\neq 0\}$. An argument similar to that used to prove $(a)$ of Lemma \ref{le:mono_to_mono_right} shows that, also in this case, $i_0 = 0$.

	\vspace*{2mm}
	\noindent (b) Again we use the same strategy as in Lemma \ref{le:mono_to_mono_right}.
	\begin{enumerate}[label=(\roman*)]
		\item When $k=1$ and $m>k+1=2$; since $\pi_\cS(\psi(X))\in \pi_\cS(\cH_{1,s}(\tilde{\eta}, \tilde{h})),$ we may assume, by way of contradiction, that there exist $b_0$, $b_1\in \F_{q^n}$ with $b_1\neq 0$ such that 
		\[\psi(X) \equiv b_0X + b_1 X^{q^s} \mod \theta_\cS.\]
		It follows that
		\[(X + \eta X^{q^s})\circ \psi(X) =\psi(X) + \eta \psi(X)^{q^s}\equiv b_0X + (b_1 + \eta b_0^{q^s}) X^{q^s} + \eta b_1^{q^s} X^{q^{2s}}\mod \theta_\cS.\]
		However, by Corollary \ref{coro:representation_all} and the assumption $m>2$, cannot exist any $a_0\in \F_{q^n}$ such that
		\[a_0X +  \tilde{\eta} a_0^{q^{\tilde{h}}}X^{q^{s}} \equiv  b_0X + (b_1 + \eta b_0^{q^s}) X^{q^s} + \eta b_1^{q^s}a^{q^{2s}} X^{q^{2s}}\mod \theta_\cS,\]
		which leads to a contradiction.
		\item 	When $k>2$ and $m>k+1$, by an analogous argument as in the proof of (a), we can show that 
		\begin{equation}\label{eq:apsi(x)_2}
			\psi(X)^{q^s}\equiv c_{0} X + c_1 X^{q^{s}} + c_2 X^{q^{2s}}  \mod \theta_\cS,
		\end{equation}
		for certain $c_i\in \F_{q^n}$.
		
		We proceed to show that $c_0$ must be $0$ by way of contradiction. It can be similarly shown that $c_2=0$, and we omit its proof.
		
		As $k>2$, we have $c_1 X^{q^{s}}, c_2X^{q^{2s}}\in  \cH_{k,s}(\tilde{\eta},\tilde{h})$. Together with \eqref{eq:apsi(x)_2}, i.e.,
		\[c_{0} X\equiv \psi(X)^{q^s}-c_1 X^{q^{s}} - c_2 X^{q^{2s}}  \mod \theta_\cS ,\]
		we have $\pi_\cS(c_{0} X) \in \pi_\cS(\cH_{k,s}(\tilde{\eta}, \tilde{h}))$.
		By Corollary \ref{coro:representation_all}, for arbitrary $g\in \cH_{k,s}(\tilde{\eta}, \tilde{h})$, $c_{0} X$ and $g$ belong to distinct residue classes in $ \lp{n}{q}[X] / (\theta_\cS)$, which leads to a contradiction.
		
		Thus $\psi(X)^{q^s}\equiv c_1 X^{q^{s}} \mod \theta_\cS$ which means $\psi(X) \equiv c_1^{q^{n-s}}X \mod \theta_\cS$, and we complete the proof.
		
		\item 	When $k=2$ and $m>k+2=4$, it is obvious that  \eqref{eq:apsi(x)_2} holds for certain $c_i\in \F_{q^n}$. We proceed to prove that $c_0$ must be $0$ by way of contradiction and it can be similarly shown that $c_2=0$. As $\psi(X)^{q^s}\equiv c_1 X^{q^{s}} \mod \theta_\cS$ implies $\psi(X) \equiv c_1^{q^{n-s}}X \mod \theta_\cS$, we complete the proof.
		
		Assume by way of contradiction that $c_0\neq 0$. By calculation, we have
		\[\psi(X) + \eta \psi(X)^{q^{2s}} \equiv c_0^{q^{n-s}} X^{q^{n-s}} + d_0 X + d_1X^{q^{s}} +d_2 X^{q^{2s}}+d_3 X^{q^{3s}}\mod \theta_\cS,\]
		where $d_0, d_1$, $d_2$ and $d_3$ are uniquely determined by $c_0,c_1$, $c_2$ and $\eta$. Let $f$ denote $\psi(X) + \eta \psi(X)^{q^{2s}}$. As $\pi_\cS(f)\in \pi_\cS(\cH_{k,s}(\tilde{\eta}, \tilde{h}))$, there exists a polynomial $g\in \cH_{k,s}(\tilde{\eta}, \tilde{h})$ such that
		\[f - g \equiv 0 \mod{\theta_\cS}.\]
		However,  as $(f - g)^{q^s}$ modulo $\theta_\cS$ is congruent to a polynomial of degree smaller than or equal to $q^{(k+2)s}<q^{ms}$,  $f-g$ has at most $q^{k+2}$ roots in $\U_\cS$ if $f-g\neq 0$, which actually holds because of $c_0\neq 0$. It contradicts the fact that $\theta_\cS$ has $q^m$ roots and $\theta_\cS \mid (f-g)$.\qedhere
	\end{enumerate}
\end{proof}
If we let $\eta=\tilde{\eta}$ and $h=\tilde{h}$, then Lemma \ref{le:mono_to_mono_middle} shows us that $\psi\in \cN_m(\pi_\cS(\cH_{k,s}(\eta, h)))$ satisfies $\psi(X) \equiv bX \mod \theta_\cS$ for some $b\in \F_{q^n}$ under certain assumptions on $m$, $k$, $h$ and $\eta$.

When $\eta \neq 0$ and $m=k+1$ or $(m,k)=(4,2)$, under certain conditions on $h$ and $s$, we can also prove the same result as in Lemma \ref{le:mono_to_mono_middle} through more complicated calculations.

\begin{lemma}\label{le:mono_to_mono_middle_conditioned}
	Let $k$, $m$ and $n$ be positive integers satisfying $k<m\leqslant n$, and let $\eta, \tilde{\eta}\in \F_{q^n}^*$ and $h,\tilde{h}\in \{0,\cdots n-1\}$ be such that $\cH_{k,s}(\eta, h)$ and $\cH_{k,s}(\tilde{\eta}, \tilde{h})$ are both generalized twisted Gabidulin codes.  Let $\cS $ be an $m$-subset of $\F_q$-linearly independent elements in $\F_{q^n}$. Assume that $\psi\in \End_{\F_{q}}(\F_{q^n})$ satisfies  $\pi_\cS(f\circ \psi )\in \pi_\cS(\cH_{k,s}(\tilde{\eta}, \tilde{h}))$ for all $f\in \cH_{k,s}(\eta, h)$. Then there exists an element $b\in \F_{q^n}$ such that
	$$\psi(X) \equiv bX \mod \theta_\cS,$$
	if one of the following collections of conditions are satisfied.
	\begin{enumerate}[label=(\alph*)]
	\item $k=1$, $m=2$ and $\tilde{h}+h\not\equiv 0,\,\tilde{h} \mod n$.
	\item $k=2$ and $\tilde{h}\neq 0$.
	\item $k=2$, $m=n=4$ and $\tilde{\eta}^{q^{2s}+1}\eta^{q^{3s}+q^s}\neq 1$. In particular, $\tilde{\eta}\neq \eta$.
	\item $k>2$, $m=k+1$ and $\tilde{h}\neq 0$.
	\end{enumerate}
\end{lemma}
\begin{proof}
	(a) By Corollary \ref{coro:representation_all}, we can assume that
	\[X^{q^{2s}} \equiv \beta_1 X^{q^s} + \beta_0 X \mod \theta_\cS.\]
	Here $\beta_0\neq 0$, otherwise $\theta_\cS \mid (X^{q^{s}} - \beta_1^{q^{n-s}} X)^{q^s}$ implies $\theta_\cS \mid (X^{q^{s}} - \beta_1^{q^{n-s}} X)$ which contradicts Corollary \ref{coro:representation_all}.
	
	Assume that $\psi(X)\equiv c_0 X + c_1 X^{q^s} \mod \theta_\cS$. This implies that
	\begin{align*}
		&a\psi(X) + \eta a^{q^h}\psi(X)^{q^s} \\
		\equiv &(ac_0X + ac_1X^{q^s}) +  (\eta a^{q^{h}}c_0^{q^s} X^{q^s} + \eta a^{q^{h}}c_1^{q^s} X^{q^{2s}}) \mod \theta_\cS\\
		\equiv &ac_0X + (ac_1 +  \eta a^{q^{h}}c_0^{q^s}) X^{q^s} + \eta a^{q^{h}}c_1^{q^s} X^{q^{2s}}\mod \theta_\cS\\
		\equiv& (ac_0 + \eta a^{q^{h}}c_1^{q^s}\beta_0)X + (ac_1 +\eta a^{q^{h}}c_0^{q^s} + \eta a^{q^{h}}c_1^{q^s}\beta_1)X^{q^s} \mod \theta_\cS,
	\end{align*}
	for every $a\in \F_{q^n}$.  As $\pi_\cS(a\psi(X) + \eta a^{q^h}\psi(X)^{q^s}))\in \pi_\cS(\cH_{1,s}(\tilde{\eta}, \tilde{h}))$, we have
	\[	\tilde{\eta}(ac_0 + \eta a^{q^{h}}c_1^{q^s}\beta_0)^{q^{\tilde{h}}}  = ac_1 +\eta a^{q^{h}}(c_0^{q^s} + c_1^{q^s}\beta_1),\]
	i.e.
	\begin{equation}\label{eq:middle_m=2_k=1}
		\tilde{\eta}c_0^{q^{\tilde{h}}}a^{q^{\tilde{h}}} + \tilde{\eta}(\eta c_1^{q^s}\beta_0)^{q^{\tilde{h}}}a^{q^{h+\tilde{h}}}= c_1a +\eta (c_0^{q^s} + c_1^{q^s}\beta_1)a^{q^{h}}.
	\end{equation}
	As $\beta_0\neq 0$ and $\tilde{h}+h\not\equiv 0,\,h \mod n$, \eqref{eq:middle_m=2_k=1} holds only if $c_1=0$. 

	\vspace*{2mm}
	\noindent(b) Now $k=2$ and we assume that
	\begin{equation}\label{eq:middle_k=2_m=4_n=4_psi}
		\psi(X)^{q^s}\equiv \sum_{i=0}^{2}d_{i} X^{q^{is}}\mod \theta_\cS.
	\end{equation}
	For each $a\in \F_{q^n}$,
	\[a\psi(X)^{q^s}\equiv \sum_{i=0}^{2}ad_{i} X^{q^{is}}\mod \theta_\cS.\]
	As $\pi_\cS(a\psi(X)^{q^{s}})$ always belongs to $\pi_\cS(\cH_{k,s}(\tilde{\eta}, \tilde{h}))$, we have
	\[\tilde{\eta}(ad_0)^{q^{\tilde{h}}}= a d_2,\]
	which holds for every $a\in \F_{q^n}$. If at least one of $d_0$ and $d_2$ is zero, then the other one also must be zero. On the other hand, if $d_0$ and $d_2$ both are nonzero, then it must be $\tilde{h}\equiv 0 \mod n$, which is already excluded in our assumption. Therefore, $d_0=d_2=0$ and $\psi(X)^{q^s}\equiv d_1 X^{q^s}\mod \theta_\cS$.
	
	\vspace*{2mm}
	\noindent(c) Following the proof of $(b)$, we only have to consider the case $\tilde{h}=0$. 
	
	Instead of looking at elements in 
	\[\Psi := \{\psi\in \End_{\F_{q}}(\F_{q^4}):\pi_\cS(f\circ \psi )\in \pi_\cS(\cH_{2,s}(\tilde{\eta}, 0 ))\text{ for all }f\in \cH_{2,s}(\eta, h)\},\]
	we are going to consider 
	\[\Psi' := \{\psi\in \End_{\F_{q}}(\F_{q^4}):\pi_\cS(f\circ \psi )\in \pi_\cS(\cH_{2,s}(\tilde{\eta}, 0 )^{q^s})\text{ for all }f\in \cH_{2,s}(\eta, h)^{q^s}\},\] 
	where
	\[\cH_{2,s}(\eta, h)^{q^s} := \{(f(X))^{q^s}: f\in \cH_{2,s}(\eta, h)\}=\{ a_0 X^{q^s} + a_1 X^{q^{2s}}+ \eta^{q^s} a_0^{q^h} X^{q^{3s}}: a_0,a_1\in \F_{q^4} \}.\]
	Let $\sigma(x):= x^{q^s}$ for $x\in \F_{q^4}$. Clearly the map from $\Psi'$ to $\Psi$ defined by $\psi \mapsto \sigma^{-1} \circ \psi \circ \sigma$ is a bijection. 
	If we can show that for each element $\psi\in\Psi'$, there exists  $b\in \F_{q^4}$ such that	$\psi(X) \equiv bX \mod \theta_\cS$, then we complete the proof.

	Let $\psi$ be defined as in \eqref{eq:middle_k=2_m=4_n=4_psi}. As $m=n=4$, it is clear that $\theta_\cS = X^{q^4}-X$. By calculation,  we have
	\begin{align*}
		 &\psi(X)^{q^s} + \eta^{q^s} \psi(X)^{q^{3s}}\\
		=&\sum_{i=0}^{2}d_{i} X^{q^{is}}+ \eta^{q^s} \left(\sum_{i=0}^{2}d_{i} X^{q^{is}}\right)^{q^{2s}} \mod \theta_\cS\\
		=&( d_0 + \eta^{q^s}   d_2^{q^{2s}})X + \sum_{i=1}^{3}e_i X^{q^{is}}  \mod \theta_\cS,
	\end{align*}
	where the precise value of $e_i$ for $i=1,2,3$ is not required in the rest of our proof. As $\pi_\cS( \psi(X)^{q^s} + \eta^{q^s}  \psi(X)^{q^{3s}}))\in \pi_\cS(\cH_{2,s}(\tilde{\eta}, 0)^{q^s})$, we have
	\begin{equation}\label{eq:middle_k=2_m=n=4_1}
		d_0 + \eta^{q^s} d_2^{q^{2s}}=0.
	\end{equation}
	On the other hand, 	since $\pi_\cS( \psi(X)^{q^{2s}})$ always belongs to $\pi_\cS(\cH_{2,s}(\tilde{\eta}, 0)^{q^s})$ and
	\begin{align*}
		 \psi(X)^{q^{2s}}\equiv&  \left(\sum_{i=0}^{2}d_{i} X^{q^{is}}\right)^{q^{s}}\mod \theta_\cS\\
		\equiv&  d_0^{q^s}X^{q^s} +  d_1^{q^{s}}X^{q^{2s}} +  d_2^{q^s}X^{q^{3s}}\mod \theta_\cS,
	\end{align*}
	we have
		\[\tilde{\eta}^{q^s} d_0^{q^s}= d_2^{q^s}.\]
	Together with \eqref{eq:middle_k=2_m=n=4_1}, we have
	\[ d_0^{q^{2s}-1} = -\frac{1}{\tilde{\eta}^{q^{2s}}\eta^{q^s} }. \]
	Taking the $q^{2s}+1$-th power of its both sides, we obtain
	\begin{equation}\label{eq:tilde_eta_eta_norm}
		\tilde{\eta}^{q^{2s}+1}\eta^{q^{3s}+q^s}=1.
	\end{equation}
	It contradicts the assumption.
	
	In particular, if $\tilde{\eta}= \eta$, then \eqref{eq:tilde_eta_eta_norm} implies that $N_{q^{4s}/q^s}(\eta)=1$ which contradicts the condition on $\eta$.
	
	\vspace*{2mm}
	\noindent(d) Now $k>2$. First we show that $\psi(X)^{q^s}\equiv d_{0}X + d_1 X^{q^{s}}\mod \theta_\cS$.
	
	By way of contradiction, we assume that
	\[\psi(X)^{q^s}\equiv \sum_{i=0}^{i_0}d_{i} X^{q^{is}}\mod \theta_\cS,\]
	where $2\leqslant i_0 \leqslant m-1$ and $d_{i_0}\neq 0$. Hence
	\begin{align*}
		a\psi(X)^{q^{(m-i_0+1)s}} \equiv & a\sum_{i=0}^{i_0-1}d_{i}^{q^{(m-i_0)s}} X^{q^{(i+m-i_0)s}} +  ad_{i_0}^{q^{(m-i_0)s}} X^{q^{ms}}\mod \theta_\cS.
	\end{align*}
	By Corollary \ref{coro:representation_all}, we can assume that
	\[X^{q^{ms}} \equiv \sum_{i=0}^{m-1} \beta_i X^{q^{is}} \mod \theta_\cS.\]
	Here $\beta_0\neq0$, otherwise $\theta_\cS \mid( X^{q^{(m-1)s}} - \sum_{i=1}^{m-1} \beta_i^{q^{n-s}} X^{q^{(i-1)s}} )$ which contradicts Corollary \ref{coro:representation_all}. Thus
	\begin{align*}
		a\psi(X)^{q^{(m-i_0+1)s}} \equiv & a\sum_{i=0}^{i_0-1}d_{i}^{q^{(m-i_0)s}} X^{q^{(i+m-i_0)s}} +  ad_{i_0}^{q^{(m-i_0)s}}\sum_{i=0}^{m-1} \beta_i X^{q^{is}}  \mod \theta_\cS\\
		\equiv & ad_{i_0}^{q^{(m-i_0)s}} \beta_0 X+a\sum_{i=0}^{i_0-2}d_{i}^{q^{(m-i_0)s}} X^{q^{(i+m-i_0)s}}\\
		& +ad_{i_0}^{q^{(m-i_0)s}}\sum_{i=1}^{m-2} \beta_i X^{q^{is}} + ad_{i_0-1}^{q^{(m-i_0)s}}(\beta_{m-1}+1) X^{q^{(m-1)s}} 
		\mod \theta_\cS.
	\end{align*}

	Recall that $m=k+1$. As $\pi(a\psi(X)^{q^{(m-i_0+1)s}})\in \pi_\cS(\cH_{k,s}(\tilde{\eta},\tilde{h}))$, we have
	\[ \tilde{\eta} (ad_{i_0}^{q^{(m-i_0)s}} \beta_0)^{q^{\tilde{h}}} = ad_{i_0-1}^{q^{(m-i_0)s}}(\beta_{m-1}+1)\]
	for all $a\in \F_{q^n}$. However, as $\tilde{h}\not\equiv 0\mod n$, the equation above holds for all $a$ if and only if $d_{i_0}=d_{i_0-1}=0$, which contradicts our assumption on the value of $d_{i_0}$. Hence $i_0\leq 1$, which means
	\[d_{0}X\equiv  \psi(X)^{q^s}- d_1 X^{q^{s}}\mod \theta_\cS.\]
	As $d_1 X^{q^s}\in \cH_{k,s}(\tilde{\eta},\tilde{h})$, $\pi(d_0X)\in \pi_{\cS}( \cH_{k,s}(\tilde{\eta},\tilde{h}))$. By Corollary \ref{coro:representation_all}, if $d_0\neq 0$, we get a contradiction.
	
	Therefore $\psi(X)^{q^s}\equiv  d_1 X^{q^{s}}\mod \theta_\cS$ which finishes the proof.
\end{proof}
If we let $\eta=\tilde{\eta}$ and $h=\tilde{h}$, then Lemma \ref{le:mono_to_mono_middle_conditioned} shows us the property of $\psi\in \cN_m(\pi_\cS(\cH_{k,s}(\eta, h)))$ with a few exceptions as Lemma \ref{le:mono_to_mono_middle}.

Now we can calculate the middle (right) nucleus of $ \pi_\cS(\cH_{k,s}(\eta,h))$.
\begin{theorem}\label{th:middle_nucleus}
	Let $k$, $m$ and $n$  be positive integers satisfying $k<m\leqslant n$. Let $\cS=\{\alpha_1, \alpha_2, \cdots, \alpha_m\} $ be a subset of $\F_q$-linearly independent elements in $\F_{q^n}$. Let $\F_{q^\ell}$ be the largest field such that $U_\cS$ is an $\F_{q^\ell}$-linear space. 
	\begin{enumerate}[label=(\alph*)]
	\item The middle nucleus of $ \pi_\cS(\cG_{k,s})$ is
	\[\cN_m( \pi_\cS(\cG_{k,s}))=\{cX : c\in \F_{q^\ell}\}.\]
	\item Assume $\eta\neq 0$ and at least one of the conditions in Lemmas \ref{le:mono_to_mono_middle} and \ref{le:mono_to_mono_middle_conditioned} is satisfied for $\tilde{h}=h$ and $\tilde{\eta}=\eta$. Then
	\[\cN_m( \pi_\cS(\cH_{k,s}(\eta,h)))=\{cX : c\in \F_{q^t}\},\]
	where $t = \gcd(n, sk-h, \ell)$.
	\end{enumerate}
\end{theorem}
\begin{proof}
	By Lemma \ref{le:mono_to_mono_middle}, for each $\psi \in \cN_m( \pi_\cS(\cG_{k,s}))$, 
	\[\psi(X) \equiv bX \mod \theta_\cS,\]
	for a certain $b\in \F_{q^n}$. By Lemma \ref{le:nucleus_polynomial}, two maps $\psi$ and $\psi'$ which both map $U_\cS$ to itself define the same element in $\cN_m( \pi_\cS(\rC))$ if and only if $\psi |_{U_\cS}=\psi' |_{U_\cS}$, which is equivalent to 
	\[ \psi \equiv \psi' \mod \theta_\cS. \]
	Hence we only have to consider the value of $b$ when $\psi$ maps $U_\cS$ to itself, where $\psi(X)= bX$.
	\begin{enumerate}[label=(\alph*)]
	\item When $\eta=0$, it is clear that $\psi(U_\cS)\subseteq U_\cS$ if and only if $b\in \F_{q^\ell}$.
	\item When $\eta\neq 0$, by looking at 
	\[\psi(X) + \eta \psi(X)^{q^{sk}} = bX + \eta b^{q^{sk}}X^{q^{sk}}\in \cH_{k,s}(\eta,h)\]
	from $\psi(U_\cS)\subseteq U_\cS$ we can derive $b\in \F_{q^\ell}$ and $b^{q^{sk}}=b^{q^h}$, i.e., $b\in \F_{q^{\gcd(sk-h, n)}}$. Hence $b\in \F_{q^t}$. It is easy to verify that for every $b\in \F_{q^t}$, $\psi|_{U_\cS}\in \cN_m( \pi_\cS(\cH_{k,s}(\eta,h)))$. \qedhere
	\end{enumerate}
\end{proof}

\begin{theorem}\label{th:right_nucleus}
	Let $k$, $m$ and $n$ be positive integers satisfying $k<m\leqslant n$. Let $\cS=\{\alpha_1, \alpha_2, \cdots, \alpha_m\} $ be a subset of $\F_q$-linearly independent elements in $\F_{q^n}$, where $\alpha_1=1$. Let $\ell$ be the smallest integer such that $\theta_\cS \mid (X^{q^\ell}-X)$ and $r=n/\ell$.
	\begin{enumerate}[label=(\alph*)]
	\item The right nucleus of $ \pi_\cS(\cG_{k,s})$ is
	\[\cN_r( \pi_\cS(\cG_{k,s}))=\left\{\sum_{i=0}^{r-1} c_i X^{q^{i\ell}}: c_i\in \F_{q^n}\right\}.\]
	\item Assume $\eta\neq 0$ and the conditions in Lemmas \ref{le:mono_to_mono_right} hold. The right nucleus of $ \pi_\cS(\cH_{k,s}(\eta,h))$ is 
	\[\cN_r( \pi_\cS(\cH_{k,s}(\eta,h)))=\left\{\sum_{i=0}^{r-1} c_i X^{q^{i\ell}}: c_i\in \F_{q^n}\text{ and } \eta c_i^{q^h}= \eta^{q^{i\ell}} c_i\right\}.\]
	\end{enumerate}
\end{theorem}

In Theorem \ref{th:right_nucleus}, it is not difficult to see that $\ell$ always divides $n$, because $\cS\subseteq\F_{q^n}$. In fact, $\F_{q^\ell}$ is the smallest subfield of $\F_{q^n}$ containing $\cS$. Theorem \ref{th:right_nucleus} (a) is originally proved in \cite[Theorem 4.5]{liebhold_automorphism_2016} in a different language. Here we need an alternative proof of it in the form of linearized polynomials to show (b). 

It bears remarking that when $1\notin\cS$, we can still determine the right nucleus in Theorem \ref{th:right_nucleus}: We can simply take any element $\alpha\in \cS$ and replace $\cS$ by $\widetilde{\cS}:=\{c/\alpha : c\in \cS \}$, from which it follows that the new code $\pi_{\widetilde{\cS}}(\cH_{k,s}(\eta,h))$ is equivalent to $\pi_\cS(\cH_{k,s}(\eta,h))$. Hence by calculating  $\cN_r( \pi_{\widetilde{\cS}}(\cH_{k,s}(\eta,h)))$, we determine $\cN_r( \pi_\cS(\cH_{k,s}(\eta,h)))$.
\begin{proof}
	(a) First, it is easy to see that for any $c\in\F_{q^n}$, the map $\varphi$ defined by $\varphi: aX\mapsto caX$ is in $\cN_r( \pi_\cS(\cG_{k,s}))$. According to Lemma \ref{le:liebhold}, $\cN_r( \pi_\cS(\cG_{k,s}))\cong \F_{q^{\ell'}}^{r'\times r'}$ for a subfield $\F_{q^{\ell'}}$ of $\F_{q^n}$ and $r'=n/\ell'$. Next we show that $\ell'=\ell$.
	
	Define
	\[ \cT := \left\{\sum_{i=0}^{r-1}  c_i X^{q^{i\ell}}: c_i\in \F_{q^n}  \right\}.\]
	As the elements in $\cT$ are $q^\ell$-polynomials, by choosing a basis of $\F_{q^n}$ over $\F_{q^\ell}$, it is not difficult to see that each element of $\cT$ defines a matrix in $\F_{q^{\ell}}^{r\times r}$ and $\cT \cong \F_{q^{\ell}}^{r\times r}$ as vector spaces over $\F_{q^\ell}$, which means that we can use some polynomials in $\cT$ to represent elements in $\cN_r( \pi_\cS(\cG_{k,s}))$.
	
	For any $\varphi=\sum_{i=0}^{r-1}  c_i X^{q^{il}}\in \cT$ and $a\in \F_{q^n}$, 
	\[\varphi(aX) = \sum_{i=0}^{r-1}  c_i (aX)^{q^{i\ell}}\equiv \left(\sum_{i=0}^{r-1}  c_i a^{q^{i\ell}}\right)X\mod \theta_\cS,\]
	because $X^{q^{\ell}}\equiv X\mod \theta_\cS$. It follows that
	\[\pi_\cS(\varphi( a_0 X + a_1 X^{q^s} + \dots +a_{k-1} X^{q^{s(k-1)}} ))\in  \pi_\cS(\cG_{k,s}),\]
	which means $\cT \subseteq \cN_r( \pi_\cS(\cG_{k,s}))$, i.e., $\ell'\leqslant \ell$.
	
	By way of contradiction, we assume that $\ell'<\ell$. This means 
	\[ \cT\subsetneq\cT' := \left\{\sum_{i=0}^{r'-1}  c_i X^{q^{i\ell'}}: c_i\in \F_{q^n}  \right\}=\cN_r( \pi_\cS(\cG_{k,s})) ,\]
	where the last equality comes from the fact that each matrix in $\F_{q^{\ell'}}^{r'\times r'}$ can be uniquely represented by a polynomial in $\cT'$ and $\cN_r( \pi_\cS(\cG_{k,s}))\cong \F_{q^{\ell'}}^{r'\times r'}$ which has been proved in the very beginning.
	
	Take $\varphi': X \mapsto X^{q^{\ell'}} \in \cT'$. By Lemma \ref{le:mono_to_mono_right}, there is a $w\in \F_{q^n}$ such that
	\[\varphi'(X)=X^{q^{\ell'}} \equiv wX \mod \theta_\cS,\]
	which means $\theta_\cS \mid (X^{q^{\ell'}}-wX)$. As $1\in \cS$ and $\theta_\cS \mid (X^{q^{\ell'}}-wX)$, we can derive $1-w=0$ whence
	$\theta_\cS \mid (X^{q^{\ell'}}-X)$. It contradicts the minimality of $\ell$.
	
	\vspace*{2mm}
	\noindent(b) Let $\varphi\in \cN_r( \pi_\cS(\cH_{k,s}(\eta,h)))$. By Lemmas  \ref{le:mono_to_mono_right} (b), for any $a\in \F_{q^n}$ there exists an element $b\in \F_{q^n}$ such that $\varphi(aX) \equiv bX \mod \theta_\cS$. By Lemma \ref{le:extra_important}, we see that $\varphi(aX^{q^{is}})\equiv b_i X^{q^{is}} \mod \theta_\cS$ for some $b_i\in \F_{q^n}$. Thus by Lemma \ref{le:nucleus_polynomial}, $\varphi$ also defines an element in $\cN_r( \pi_\cS(\cG_{k,s}))$. Hence, by (a), we have
	\[\cN_r( \pi_\cS(\cH_{k,s}(\eta,h)))\subseteq\left\{\sum_{i=0}^{r-1} c_i X^{q^{i\ell}}: c_i\in \F_{q^n}\right\}.\] 
	
	Now let us verify which $\varphi\in \{\sum_{i=0}^{r-1} c_i X^{q^{i\ell}}: c_i\in \F_{q^n}\}$ belongs to $\cN_r( \pi_\cS(\cH_{k,s}(\eta,h)))$. For any $a_0\in \F_{q^n}$,
	\[\varphi(a_0X)= \sum_{i=0}^{r-1}  c_i (a_0X)^{q^{i\ell}}\equiv \left(\sum_{i=0}^{r-1}  c_i a_0^{q^{i\ell}}\right)X\mod \theta_\cS,\]
	and
	\[\varphi(\eta a_0^{q^h}X^{q^{sk}}) \equiv \left(\sum_{i=0}^{r-1}  c_i \eta^{q^{i\ell}}a_0^{q^{h+i\ell}}\right)X^{q^{sk}}\mod \theta_\cS.\]
	From the two above equations, we get 
	$$\varphi(a_0X+\eta a_0^{q^h}X^{q^{sk}}) \equiv \left(\sum_{i=0}^{r-1}  c_i a_0^{q^{i\ell}}\right)X + \left(\sum_{i=0}^{r-1}  c_i \eta^{q^{i\ell}}a_0^{q^{h+i\ell}}\right)X^{q^{sk}}  \mod \theta_\cS.$$ 
	In order to have $\varphi \in \cN_r( \pi_\cS(\cH_{k,s}(\eta,h)))$, we must have
	\[ \eta \left(\sum_{i=0}^{r-1}  c_i a_0^{q^{i\ell}}\right)^{q^h} =  \sum_{i=0}^{r-1}  c_i \eta^{q^{i\ell}}a_0^{q^{h+i\ell}},\]
	for every $a_0\in \F_{q^n}$, which means
	\[ \sum_{i=0}^{r-1}\left(\eta c_i^{q^h} - \eta^{q^{i\ell}} c_i\right)a_0^{q^{h+i\ell}}=0,\]
	for every $a_0\in \F_{q^n}$. Hence $\eta c_i^{q^h}= \eta^{q^{i\ell}} c_i$, which concludes the proof.
\end{proof}
\begin{remark}
	Theorem \ref{th:middle_nucleus} (a) was first proved in \cite{morrison_equivalence_2014}; see \cite[Lemma 4.1]{liebhold_automorphism_2016} too. Theorem \ref{th:right_nucleus} (a) was first proved in \cite{liebhold_automorphism_2016} where the matrices in the MRD code (see \eqref{eq:general_code_from_poly}) are all transposed. Thus the left idealiser (resp.\ right idealiser) found in \cite{liebhold_automorphism_2016} corresponds to the right nucleus (resp.\ middle idealiser) here.
\end{remark}

In the end, we summarize the value of $m$ and those parameters of $\cH_{k,s}(\eta,h)$ for which we cannot determine its middle or right nucleus. For the following cases, the right nucleus of $\pi_\cS(\cH_{k,s}(\eta,h))$ is still unknown:
\begin{itemize}
	\item $m=k+1$;
	\item $m=4$ and $k=2$.
\end{itemize}
As $\cG_{1,s}$ and $\cH_{1,s}(\eta,0)$ are equivalent (see \cite{biliotti_collineation_1999}), we can exclude the case $k=1$ and $h\equiv 0 \mod n$ for the middle nucleus of $\pi_\cS(\cH_{k,s}(\eta,h))$. However, it is still unknown for the following cases:
\begin{itemize}
\item $k=1$, $m=2$ and $n=2h$;
\item $k=2$, $m=3$ and $h=0$;
\item $k=2$, $m=4$ and $n>m$;
\item $k>2$, $m=k+1$ and $h=0$.
\end{itemize}

\section{Automorphism groups of Generalized twisted Gabidulin codes}
When $m=n$, the automorphism group of any generalized twisted Gabidulin code has been completely determined in \cite{lunardon_generalized_2018}. More precisely, let $(\varphi_1, \varphi_2, \rho)$ be in $\Aut(\cH_{k,s}(\eta,h))$, it was shown that $\varphi_1$ and $\varphi_2$ must be monomials over $\F_{q^n}$. In this section, we proceed to show an analogous result for the case $m<n$.

Let $\cN_r (\pi_\cS(\cH_{k,s}(\eta,h)))$ be the right nucleus of $\pi_\cS(\cH_{k,s}(\eta,h))$ determined in Theorem \ref{th:right_nucleus} and we denote it by $\cN_r$ for short. For $\rho\in \Aut(\F_q)$, by Theorem \ref{th:right_nucleus}, $\cN_r^\rho$ is the right nucleus of $\pi_\cS(\cH_{k,s}(\eta^\rho,h))$. 

For any $(\varphi_1, \varphi_2, \rho)\in \Aut(\pi_\cS(\cH_{k,s}(\eta,h)))$, it is routine to verify that $\gamma \mapsto \varphi_1 \circ \gamma^\rho \circ \varphi_1^{-1}$ is an automorphism on $\cN_r$; see \cite[Lemma 2.5]{schmidt_number_MRD_2018}. Equivalently, the map given by $\gamma'\mapsto \varphi_1 \circ \gamma' \circ \varphi_1^{-1}$ is an isomorphism from $\cN_r^\rho$ to $\cN_r$.

We define $\Theta$ by
\[\Theta :=  \left\{\sum_{i=0}^{r-1} c_i : \sum_{i=0}^{r-1} c_i X^{q^{i\ell}}\in \cN_r\right\}.\]

\begin{lemma}\label{le:necessary_normalizer}
	Let $k$, $m$ and $n$ be positive integers satisfying $k<m<n$ and let $\cS$, $\ell$, $r$ and $\cN_r$ be determined as in Theorem \ref{th:right_nucleus}. Let $\lambda$ be a linearized polynomial such that $\lambda$ is invertible and $\lambda \circ \varphi^\rho \circ \lambda^{-1}\in \cN_r$ for any $\varphi\in \cN_r$. 
	\begin{enumerate}[label=(\alph*)]
	\item\label{item:assumption_c_1} If $\Theta \bigcap (\F_{q^\ell}\setminus \F_q) \neq \emptyset$, then there exists $b\in \F_{q^n}^*$ and $u\in \{0,1,\cdots,\ell-1  \}$  such that
		\[\lambda\equiv bX^{q^u} \mod X^{q^\ell}-X.\]
	\item\label{item:assumption_c_2} If $\Theta =\F_{q^n}$, then there exists $b_i\in \F_{q^n}$ for $i=0,1,\cdots, r-1$ and $u\in \{0,1,\cdots,\ell-1  \}$  such that
			\[\lambda= \sum_{i=0}^{r-1}b_i X^{q^{u+i\ell}}.\]
	\end{enumerate}
\end{lemma}
\begin{proof}
	(a) According to the assumption and Theorem \ref{th:right_nucleus}, there exists $\varphi\in \cN_r$ such that $\varphi\equiv cX \mod X^{q^\ell}-X$ where $c\in \F_{q^\ell}\setminus\F_q$. Assume that $\lambda \equiv \sum_{i=0}^{\ell-1} b_i X^{q^i} \mod{X^{q^\ell}-X}$. Then, it can be readily verified that
	\begin{equation}\label{eq:lambda(cX)}
		\lambda(\varphi(X)) \equiv \lambda(cX) \equiv \sum_{i=0}^{\ell-1}b_i c^{q^i}X^{q^i} \mod X^{q^\ell}-X.
	\end{equation}
	
	If $\lambda \circ \varphi^\rho \circ \lambda^{-1}=\psi\in \cN_r$, then $\lambda \circ \varphi^\rho = \psi \circ \lambda$. According to Theorem \ref{th:right_nucleus}, there exists $d\in \F_{q^n}^*$ such that $\psi \equiv dX \mod X^{q^\ell}-X$. Thus
	\begin{equation}\label{eq:dlambda(X)}
		\lambda(cX)\equiv d\lambda(X) \equiv  \sum_{i=0}^{\ell-1}b_i d X^{q^i} \mod X^{q^\ell}-X. 
	\end{equation}
	As $\lambda$ corresponds to an element in $\GL(n,q)$, there exists at least one $b_{i_0}\neq 0$. By \eqref{eq:lambda(cX)} and \eqref{eq:dlambda(X)}, we have
	\[ b_{i} c^{q^{i}} =b_{i} d,  \quad \text{for all }i. \]
	Hence $c^{q^{i_0}} = d$ and $b_i=0$ if $i\neq i_0$, because of $c\notin \F_q$. Therefore $\lambda \equiv  b_{i_0} X^{q^{i_0}}\mod X^{q^\ell}-X$.
	
	(b) As $\Theta =\F_{q^n}$, we have
	\begin{equation}\label{eq:c_all_Fqn}
		\{c : \varphi \equiv cX \mod X^{q^\ell}-X \text{ for }\varphi\in \cN_r\}=\F_{q^n}.
	\end{equation}
	Assume that $\lambda(X)=\sum_{i=0}^{n-1}a_iX^{q^i}$. For each $\varphi\in \cN_r$, it is easy to show that $\lambda(\varphi(X)) \equiv \lambda(cX) \mod X^{q^\ell}-X$ for some $c\in \F_{q^n}$, whence
	\[\lambda(\varphi(X))\equiv \lambda(cX)=\sum_{i=0}^{n-1}a_ic^{q^i}X^{q^i} \equiv \sum_{j=0}^{\ell-1}\left(\sum_{i=0}^{r-1} a_{j+i\ell}(c^{q^{j}})^{q^{i\ell} }  \right) X^{q^j}  \mod X^{q^\ell}-X.\]
	From part (a) of the proof, we know that $\sum_{i=0}^{r-1} a_{j+i\ell}(c^{q^{j}})^{q^{i\ell} }=0$ for $j\neq u$. This equation, together with \eqref{eq:c_all_Fqn}, imply that $a_{j+il}=0$ for all $i\in \{0,\cdots, r-1\}$ when $j\neq u$.
\end{proof}

The two assumptions on $\sum_{i=0}^{r-1} c_i X^{q^{i\ell}}\in \cN_r$ in Lemma \ref{le:necessary_normalizer} are always satisfied for $\eta=0$, i.e.\ $\cH_{k,s}(\eta,h)=\cG_{k,s}$. However, in general, these assumptions depend on the value of $n,s,h,\ell$ and $\eta$. For instance, when $\gcd(n,h,\ell)>1$, \ref{item:assumption_c_1} holds because $\{c_0 X: c_0\in \F_{q^n} \cap \F_{q^h}\} \subseteq \cN_r$; when $\gcd(n,h)=1$ and $\eta\in \F_q$, \ref{item:assumption_c_1} does not hold anymore.

\begin{theorem}\label{th:monomial_automorphism}
	Assume that $k<m<n$ and let $\cS$, $\ell$ and $r$ be determined as in Theorem \ref{th:right_nucleus}. Suppose that one of the following collections of conditions are satisfied:
	\begin{enumerate}[label=(\alph*)]
	\item	$\eta=0$;
	\item   $\eta\neq 0$, $m>k+1$ and $(m,k)\neq (4,2)$.
	\end{enumerate}
	If $\Theta \bigcap (\F_{q^\ell}\setminus \F_q) \neq \emptyset$, then $(\varphi_1, \varphi_2,\rho)$ defines an automorphism of $\pi_\cS(\cH_{k,s}(\eta,h))$ only if there exist $a,b\in \F_{q^n}^*$ and $u\in \{0,1,\cdots, \ell-1\}$ such that
	$$\varphi_1(X)\equiv aX^{q^u} \mod X^{q^{\ell}}-X \quad \text{ and } \quad \varphi_2(c)=bc^{q^{-u}},$$
	for $c\in U_\cS$. Moreover, if $\Theta =\F_{q^n}$, then it is necessary that
	\[\varphi_1(X)=\sum_{i=0}^{r-1}a_i X^{q^{u+i\ell}},\]
	with $a_i\in \F_{q^n}$ for $i=0,1,\cdots, r-1$.
\end{theorem}

In Theorem  \ref{th:monomial_automorphism}, Lemma \ref{le:mono_to_mono_middle} holds for $\eta=\tilde{\eta}$ under any of the assumptions (a) and (b), which together with $1\in \cS$ guarantee that Theorem \ref{th:right_nucleus} holds as well. Thus we can apply Lemma \ref{le:necessary_normalizer} in the following proof.

\begin{proof}[Proof of Theorem \ref{th:monomial_automorphism}]
	According to the discussion above Lemma \ref{le:necessary_normalizer}, for any $\gamma\in \cN_r$ and any $(\varphi_1, \varphi_2, \rho)\in \Aut(\pi_\cS(\cH_{k,s}(\eta,h)))$,  $\varphi_1 \circ \gamma^\rho \circ \varphi_1^{-1}$ belongs to $\cN_r$.
	
	As there exists $\sum_{i=0}^{r-1} c_i X^{q^{i\ell}}\in \cN_r(\pi_\cS(\cH_{k,s}(\eta,h)))$ such that $\sum_{i=0}^{r-1} c_i\in \F_{q^\ell}\setminus\F_q$, by Lemma \ref{le:necessary_normalizer} \ref{item:assumption_c_1}, we obtain
	$$\varphi_1 \equiv  aX^{q^u} \mod X^{q^{\ell}}-X$$ 
	for certain $u\in \{0,1,\cdots, \ell-1\}$ and nonzero $a\in \F_{q^n}$. Let $f$ be an arbitrary element in $\cH_{k,s}(\eta,h)$. Assume that 
	$$f^\rho = \sum_{i=0}^{k-1}c_i X^{q^{is}} + \eta^\rho c_0^{q^h}X^{q^{ks}}.$$
	By calculation,
	\begin{align*}
		\varphi_1 \circ f^\rho &\equiv  \sum_{i=0}^{k-1}(ac_i^{q^u}) X^{q^{is+u}} + \eta^{\rho q^u} (ac_0^{q^{u+h}})X^{q^{ks+u}} \mod X^{q^{\ell}}-X\\
							   &\equiv \sum_{i=0}^{k-1}(ac_i^{q^u}) (X^{q^u})^{q^{is}} + \eta^{\rho q^u} a^{1-q^h}(ac_0^{q^{u}})^{q^h}(X^{q^u})^{q^{ks}} \mod X^{q^{\ell}}-X.
	\end{align*}
	Hence $\pi_\cS(g) \in \pi_\cS(\cH_{k,s}( \eta^{\rho q^u} a^{1-q^h},h))$, where $g=\varphi_1 \circ f^\rho\circ \epsilon$ and $\epsilon =X^{q^{n-u}}\in \F_{q^n}[X]$. As $f$ is arbitrary, $g$ can be any element in $\cH_{k,s}( \eta^{\rho q^u} a^{1-q^h},h)$.
	
	As $\pi_\cS((\varphi_1 \circ f^\rho \circ \epsilon) \circ (\epsilon^{-1}\circ\varphi_2))=\pi_\cS(\varphi_1\circ f^\rho \circ \varphi_2)\in \pi_\cS(\cH_{k,s}(\eta,h))$, the map $\psi = \epsilon^{-1}\circ\varphi_2$ is such that
	\[\pi_\cS(g \circ \psi) \in \pi_\cS(\cH_{k,s}(\eta,h)),\]
	for every $g\in \cH_{k,s}( \eta^{\rho q^u} a^{1-q^h},h)$. By Lemmas \ref{le:mono_to_mono_middle} and \ref{le:mono_to_mono_middle_conditioned} (at least one of the collections of conditions there are satisfied), 
	\[\psi(X)\equiv bX \mod \theta_\cS,\]
	for certain $b\in \F_{q^n}^*$. This implies that $\psi|_{U_\cS}(c)=bc$ for all $c\in U_\cS$. Hence $\varphi_2|_{U_\cS}(c)=b^{q^{-u}}c^{q^{-u}}$.
	
	When $\Theta =\F_{q^n}$, the further result on $\varphi_1$ can be derived directly from Lemma \ref{le:necessary_normalizer} \ref{item:assumption_c_2}.
\end{proof}
\begin{remark}
	For  $\eta=0$, i.e.\ $\cH_{k,s}(\eta, h)=\cG_{k,s}$, Theorem \ref{th:monomial_automorphism} was proved in \cite{liebhold_automorphism_2016} in the form of matrices.
\end{remark}

\section*{acknowledgment}
This work is supported by the Research Project of MIUR (Italian Office for University and Research) ``Strutture geometriche, Combinatoria e loro Applicazioni" 2012. Yue Zhou is supported by the Alexander von Humboldt Foundation and the National Natural Science Foundation of China (No.\ 11771451). The authors would like to thank the anonymous referee for her/his valuable comments and suggestions on the manuscript.

\end{document}